\documentclass[12pt]{amsart}
\usepackage{amsmath,amsfonts,amssymb,amscd,amsthm,amsbsy,epsf}
\newtheorem{thm}{Theorem}[section]

\newtheorem{cor}[thm]{Corollary}
\newtheorem{prop}[thm]{Proposition}
\theoremstyle{definition} 		
\newtheorem{defn}[thm]{Definition}
\newtheorem{remark}[thm]{Remark}
\newtheorem*{note}{Notation}

\def\nat{{\mathbb N}}
\def\zat{{\mathbb Z}}
\def\qat{{\mathbb Q}}

\def\D{{\mathcal D}}

\def\L{{\mathcal L}}

\begin{document}
\title{rational dynamical systems}
\author{Andreas Koutsogiannis}

\begin{abstract}
We introduce the notion of a rational dynamical system extending the classical notion of a topological dynamical system and we prove (multiple) recurrence results for such systems via a partition theorem for the rational numbers proved by Farmaki and the author in [FK]. In particular, we extend classical recurrence results developed by Furstenberg and Weiss. Also, we give some applications of these topological recurrence results to topology, to combinatorics, to diophantine approximations and to number theory.
\end{abstract}


\maketitle
\baselineskip=18pt
\pagestyle{plain}


\section*{Introduction}

In 1927, Birkhoff proved (in [Bi]) that every topological dynamical system $(X,T),$ where $X$ is a compact metric space and $T:X\rightarrow X$ is a continuous map, has a recurrent element, which means that there exist $x\in X$ and a sequence of natural numbers $(\alpha_n)_{n\in \nat}\subseteq \nat$ with $\alpha_n\rightarrow \infty$ such that $T^{\alpha_n}(x)\rightarrow x.$ The multiple version of Birkhoff's recurrence theorem is due to Furstenberg and Weiss:

\begin{thm}([FuW], 1978)\label{thm:A}
If $X$ is a compact metric space and $T_1,\ldots,T_l$ are commuting continuous maps of $X$ to itself, then there exists a point $x\in X$ and a sequence $(\alpha_n)_{n\in \nat}\subseteq \nat,$ $\alpha_n\rightarrow \infty$ such that $T_i^{\alpha_n}(x)\rightarrow x$ simultaneously for $1\leq i\leq l$ (in this case, $x$ is said to be a \textit{multiple recurrent point} for $T_1,\ldots,T_l$).
\end{thm}

A consequence (in fact an equivalent form) of Theorem~\ref{thm:A} is the following:

\begin{thm}\label{thm:B} Let $l\in \nat$ and $\varepsilon>0.$ If $X$ is a compact metric space and $T_1,\ldots,T_l$ are commuting continuous maps of $X$ to itself, then there exists $x_0\in X$ and $n_0\in \nat,$ such that $T_i^{n_0}(x_0)\in B(x_0,\varepsilon)$ for every $1\leq i\leq l.$
\end{thm}

In fact, Theorem~\ref{thm:B} can be considered as the topological dynamics version of Gallai's partition theorem (see [GRS]), namely that for $l\in \nat$ and any finite partition of $\nat^l,$ one of the cells of the partition contains an affine image of any finite subset of $\nat^l$ (an \textit{affine image} of a finite subset $F$ of $\nat^l$ is any set of the form $\alpha+\beta F$ where $\alpha\in \nat^l,\beta\in \nat$). We note that Gallai's theorem is the multidimensional extension of van der Waerden's theorem ([vdW], 1927), that for any finite partition of the set of natural numbers there exists a cell of the partition which contains arbitrary long arithmetic progressions, which is a (perhaps the most) fundamental result in Ramsey theory.

 Our starting point for this paper is a partition theorem for the set of rational numbers (Theorem~\ref{thm:000} below) proved in [FK], which can be characterized as a strong van der Waerden theorem for the set of rational numbers. This theorem follows from a more general partition theorem for words (in [FK]), using the fact that a rational number can be represented as a word, as, according to a result of Budak-I\c{s}ik-Pym (in [BIP]), a rational number $q$ has a unique expression in the form $$q=\sum^{\infty}_{s=1}q_{-s}\frac{(-1)^{s}}{(s+1)!}\;+\;\sum^{\infty}_{r=1}q_{r}(-1)^{r+1}r! $$ where $(q_n)_{n \in \mathbb{Z}^\ast}\subseteq \nat\cup\{0\}$ with $\;0\leq q_{-s}\leq s$ for every $s>0$, $ 0\leq q_r\leq r$ for every $r> 0$ and $q_{-s}=q_r=0$ for all but finite many $r,s.$

Extending the classical notion of the topological dynamical system we introduce the notion of a rational dynamical system (Definition~\ref{d.syst}). Consequently we develop a recurrence theory for rational dynamical systems, extending the fundamental results of Furstenberg and Weiss for dynamical systems indexed by natural numbers ([Fu], [FuW]) stated above. We remark that the presented recurrence results for rational systems are stronger than those that follows from the more general recurrence results concerning topological dynamical systems indexed by words presented in [FK2]. Specifically:

(1) We obtain a topological van der Waerden theorem for the set of rational numbers (Theorem~\ref{thm:ole}) and its multiple version (Theorem~\ref{thm:olem1}) extending Theorem~\ref{thm:B} to rational dynamical systems.

(2) Introducing the minimal rational systems and characterizing them as systems having only uniformly rational recurrent points we prove, in Theorem~\ref{thm:olem2}, a strong recurrence property of minimal rational dynamical systems giving an equivalent reformulation of Theorem~\ref{thm:olem1}.

(3) We obtain a strengthening of Theorem~\ref{thm:A} for rational dynamical systems in Theorem~\ref{thm:olem4} and we prove that it is an equivalent expression of Theorems~\ref{thm:olem1} and ~\ref{thm:olem2}.

Also, we present some applications of the previously mentioned results to topology (Theorems~\ref{thm:cor2} and ~\ref{cor:olem2}), to combinatorics (Theorem~\ref{thm:n12}, Theorem~\ref{thm:olem3}, Corollary~\ref{prop:n11}), to diophantine approximations and to number theory (applications of Theorems~\ref{thm:cor2} and ~\ref{cor:olem2}). For example, as an application of Theorem~\ref{thm:olem1} we get a strong Gallai-type partition theorem for the set of rational numbers (Theorem~\ref{thm:olem3}).

\medskip

We will use the following notation.

\begin{note}
Let $\nat=\{1,2,\ldots\}$ be the set of natural numbers, $\mathbb{Z}=\{\ldots,-2,-1,0,1,2,\ldots\}$ the set of integer numbers, $\qat=\{\frac{m}{n}:m\in \zat,\;n\in \nat\}$ the set of rational numbers and $\zat^-=\{-n:n\in \nat\},$ $\zat^\ast=\zat\setminus\{0\},$ $\qat^{\ast}=\qat\setminus\{0\}.$ We denote by $[X]_{>0}^{<\omega}$ the set of all the non-empty finite subsets of $X$.

For a sequence  $(x_n)_{n\in \nat}$ of real numbers we set $\;FS\big[(x_n)_{n\in \nat}\big]=\{ \sum_{n\in F}x_n:\;F\in [\nat]_{>0}^{<\omega}\}.$
\end{note}

\section{A Topological van der Waerden-type theorem \\ for the set of rational numbers}

We introduce the notion of a simple rational dynamical system defined from a compact metric space $X$ and a sequence $\{T_n\}_{n\in \zat^\ast}$ of continuous functions from $X$ to itself (Definition~\ref{h.syst}). We prove a recurrence theorem for such systems in Theorem~\ref{thm:ole}, extending the analogous result of Furstenberg and Weiss (Theorem~\ref{thm:B}, case $l=1$). Theorem~\ref{thm:ole} is an implication of a strengthened van der Waerden-type theorem for the set of rational numbers proved in [FK] (see Theorem~\ref{thm:000} below). The inverse implication is partially correct since Theorem~\ref{thm:ole} implies a weaker form of Theorem~\ref{thm:000}, which actually can be considered as a van der Waerden-type theorem for the set of rational numbers.

According to \cite{BIP} every rational number $q$  has a unique expression in the form $$q=\sum^{\infty}_{s=1}q_{-s}\frac{(-1)^{s}}{(s+1)!}\;+\;\sum^{\infty}_{r=1}q_{r}(-1)^{r+1}r! $$ where $(q_n)_{n \in \mathbb{Z}^\ast}\subseteq \nat\cup\{0\}$ with $\;0\leq q_{-s}\leq s$ for every $s>0$, $ 0\leq q_r\leq r$ for every $r> 0$ and $q_{-s}=q_r=0$ for all but finite many $r,s$. So, for every non-zero rational number $q$, there exists a unique $l\in \nat,\;\{t_1<\ldots<t_l\}=dom(q)\in [\zat^\ast]^{<\omega}_{>0}$ and $\{q_{t_1},\ldots,q_{t_l}\}\subseteq \nat$ with $1\leq q_{t_i}\leq -t_i$ if $t_i<0$ and $1\leq q_{t_i}\leq t_{i}$ if $t_i> 0$ for every $1\leq i\leq l,$ such that defining $dom^-(q)=\{t\in dom(q):\;t<0\}$ and $dom^+(q)=\{t\in dom(q):\;t>0\}$ to have $$ q=\sum_{t\in dom^-(q)}q_t\frac{(-1)^{-t}}{(-t+1)!}\;+\;\sum_{t\in dom^+(q)}q_t(-1)^{t+1}t!\;\;(\text{we set}\;\;\sum_{t\in\emptyset}=0).$$ Observe that $$e^{-1}-1=-\sum^{\infty}_{t=1}\frac{2t-1}{(2t)!}< \sum_{t\in dom^-(q)}q_t\frac{(-1)^{-t}}{(-t+1)!} <\sum^{\infty}_{t=1}\frac{2t}{(2t+1)!}=e^{-1}$$ and $$ \sum_{t\in dom^+(q)}q_t(-1)^{t+1}t! \in \zat^\ast\;\;\text{if}\;\;dom^+(q)\neq \emptyset.$$

For two non-zero rational numbers $q_1,q_2$ we set \begin{center} $q_1\prec q_2 \;\Longleftrightarrow\; \max dom^-(q_2)<\min dom^-(q_1)<\max dom^+(q_1)<\min dom^+(q_2).$ \end{center}

Using the previous representation of rational numbers we have the following partition theorem for the rational numbers.

\begin{thm}([FK])\label{thm:000} Let $\mathbb{Q}=Q_1\cup\ldots \cup Q_r$ for $r\in \nat$. Then, there exist $1\leq i_{0}\leq r$ and for every $n\in \nat$ a function $q_n:\{1,\ldots,n\}\times\{1,\ldots,n\}\cup\{(0,0)\}\rightarrow \mathbb{Q}$ with
$$q_n(i,j)=\sum_{t\in C_n^-}q^n_t\frac{(-1)^{-t}}{(-t+1)!} + i\sum_{t\in V_n^-}\frac{(-1)^{-t}}{(-t+1)!} +  \sum_{t\in C_n^+}q^n_t(-1)^{t+1}t!+j\sum_{t\in V_n^+}(-1)^{t+1}t!,$$
\noindent where
    $C_n^-,V_n^-\in [\zat^-]^{<\omega}_{>0},\;C_n^+,V_n^+\in [\nat]^{<\omega}_{>0}$ with $C_n^-\cap V_n^-=\emptyset=C_n^+\cap V_n^+,$ $q^n_t\in \nat$ with $1\leq q_t^n\leq -t$ for $t\in C_n^-,$ $1\leq q_t^n\leq t$ for $t\in C_n^+,$ which satisfy
 $q_n(i_n,j_n)\prec q_{n+1}(i_{n+1},j_{n+1})$ for every $n\in \nat,$ and
$$FS\big[\big(q_n(i_{n},j_{n})\big)_{n\in \nat}\big]\subseteq Q_{i_0}$$ for all $((i_n,j_n))_{n\in \nat}\subseteq \nat\times\nat\cup\{(0,0)\}$ with $0\leq i_n,j_n\leq n$ for every $n\in \nat.$
\end{thm}

\begin{defn}\label{h.syst} Let $X$ be a compact metric space and $\{T_n\}_{n\in \zat^\ast}$ a family of commuting continuous maps of $X$ to itself. For every non-zero rational number $q$ with $dom(q)=\{t_1<\ldots<t_l\},$ we define $\mathcal{T}^q(x)=T_{t_1}^{q_{t_1}}\ldots T_{t_l}^{q_{t_l}}(x)$ and $\mathcal{T}^0(x)=x$ for every $x\in X.$ We say that $\mathcal{T}=(\mathcal{T}^q)_{q\in \qat}$ is a \textit{rational indexed family defined from the continuous maps} $\{T_n\}_{n\in \zat^\ast}$ of $X$ and $(X,\mathcal{T})$ is a \textit{simple rational dynamical system}.
\end{defn}

 \begin{remark}\label{rem:000} For a rational indexed family $\mathcal{T}=(\mathcal{T}^q)_{q\in \qat}$ we have in general  $\mathcal{T}^{p+q}\neq \mathcal{T}^p\mathcal{T}^q$ but if $p,q\in \qat^{\ast}$ with $q\prec p,$ then $\mathcal{T}^{p+q}= \mathcal{T}^p\mathcal{T}^q.$
\end{remark}

Using Theorem~\ref{thm:000} we can prove a recurrence result for simple rational dynamical systems extending the analogous result of Furstenberg and Weiss (Theorem~\ref{thm:B}, case $l=1$).

\begin{thm}\label{thm:ole}
Let $(X,\mathcal{T})$ be a simple rational dynamical system, $k\in \nat$ and $k_1<k_2$ be arbitrary real numbers. Then, for every $\varepsilon>0$ there exist
$\beta \in \mathbb{Q}\setminus\zat,\; \gamma\in \zat^\ast$ with $dom^+(\beta)=\emptyset=dom^-(\gamma),$ $\max dom^-(\beta)<k_1<k_2<\min dom^+(\gamma)$ and $x_0\in X$ such that
\begin{center} $\mathcal{T}^{p\beta+q\gamma}(x_0)\in B(x_0,\varepsilon)$ for every $0\leq p,q\leq k.$ \end{center}
\end{thm}

\begin{proof}
Let $k\in \nat$ and $\varepsilon>0.$ Since $X$ is compact, we have $X=\bigcup^m_{i=1}B(x_i,\frac{\varepsilon}{2})$ for some $x_1,\ldots,x_m\in X,\;m\in \nat.$ Let $x\in X.$ We form a partition of the set of rational numbers $\mathbb{Q}=Q_1\cup\ldots\cup Q_m,$ where \begin{center} $q\in Q_i\;\Leftrightarrow\;\mathcal{T}^q(x)\in B(x_i,\frac{\varepsilon}{2})$ and $\mathcal{T}^q(x)\notin B(x_j,\frac{\varepsilon}{2})$ for $j<i.$\end{center}

According to Theorem~\ref{thm:000}, there exist $1\leq i_{0}\leq m$ and $C^-,V^-,C^+,V^+$ non-empty sets with $C^-\cap V^-=\emptyset=C^+\cap V^+,$ such that $$q^\ast_{(p,q)}=\sum_{t\in C^-}q_t\frac{(-1)^{-t}}{(-t+1)!}\; + \;p\sum_{t\in V^-}\frac{(-1)^{-t}}{(-t+1)!}\; + \;\sum_{t\in C^+}q_t(-1)^{t+1}t!\;+\;q\sum_{t\in V^+}(-1)^{t+1}t!\in Q_{i_0}$$ for every $1\leq p,q\leq k+1\leq \min\{|t|:\;t\in V^-\cup V^+\},$ where $1\leq q_{t}\leq |t|,\; t\in \zat^\ast$ and $\max dom^-(q^\ast_{(1,1)})<k_1<k_2<\min dom^+(q^\ast_{(1,1)}).$ Equivalently, if \begin{center}$\beta=\sum_{t\in V^-}\frac{(-1)^{-t}}{(-t+1)!},\;\gamma=\sum_{t\in V^+}(-1)^{t+1}t!,\; \delta=\sum_{t\in C^-}q_t\frac{(-1)^{-t}}{(-t+1)!}$ and $\epsilon=\sum_{t\in C^+}q_t(-1)^{t+1}t!$\end{center} we have that $\mathcal{T}^{\delta\;+\;p\beta\;+\;\epsilon\;+\;q\gamma}(x)\in B(x_{i_0},\frac{\varepsilon}{2})$ for every $1\leq p,q\leq k+1.$

Let $x_0=\mathcal{T}^{\delta\;+\;\beta+\;\epsilon\;+\;\gamma}(x).$
Then $\mathcal{T}^{p\beta+q\gamma}(x_0)\in B(x_0,\varepsilon)$ for every $0\leq p,q\leq k.$
\end{proof}

According to the previous proof, Theorem~\ref{thm:ole} is an implication of Theorem~\ref{thm:000}. We will show that the inverse implication is partially correct. In fact we will point out that Theorem~\ref{thm:ole} implies a weaker form of Theorem~\ref{thm:000}, which actually can be considered as a van der Waerden-type theorem for the set of rational numbers. So, Theorem~\ref{thm:ole} can be considered as a topological van der Waerden theorem for the set of rational numbers.

\begin{thm}\label{thm:n12}
 Let $k_1<k_2$ be real numbers. If $\qat=Q_1\cup\ldots\cup Q_r,\;r\in \nat,$ then there exists $1\leq i_0\leq r$ such that the set $Q_{i_0}$ has the following property: for every $k\in \nat,$ there exist $\alpha\in \qat,\beta\in \qat\setminus\zat$ and $\gamma\in \zat^\ast$ with $dom^+(\beta)=\emptyset=dom^-(\gamma),$ $\max dom^-(\beta)<k_1<k_2<\min dom^+(\gamma)$ such that $\alpha+p\beta+q\gamma\subseteq Q_{i_0}$ for every $0\leq p,q\leq k.$
\end{thm}

\begin{proof}
It is sufficient to show that for every $k\in \nat$ some $Q_{j}$ satisfies the conclusion, for some $Q_{i_0}$ will do for infinite $k$ and that set will do for all $k.$

  Let $k\in \nat,\;\Omega=\{1,\ldots,r\}^{\qat}$ and enumerate $\qat=\bigcup_{n\in \nat}\{q_n\}$ with $q_1=0.$ $\Omega$ becomes compact metric space with metric \begin{center} $d(\omega,\omega')=\inf\{\frac{1}{t}:\;\omega(q_i)=\omega'(q_i)$ for $1\leq i<t\}.$ \end{center}
Let $\mathcal{T}$ be a rational indexed family which is defined from $\{T_n\}_{n\in \zat^\ast},$ where \begin{center} $T_{n}\omega(q)=\omega(q+(-1)^{n+1}n!),\;n\in \nat$ and $T_{n}\omega(q)=\omega(q+\frac{(-1)^{-n}}{(-n+1)!}),\;n\in \zat^{-}.$
\end{center}
Define a specific point $\omega\in \Omega$ according to the rule $\omega(q)=i\;\Leftrightarrow\;q\in Q_i$ and $q\notin Q_j$ for $j<i$ and let $X=\overline{\{\mathcal{T}^{s_1}\ldots \mathcal{T}^{s_m}\omega:\;s_i\in \qat,\;m\in \nat,\;1\leq i\leq m\}}.$ Then, $\mathcal{T}$ is a rational indexed family of $X.$ According to Theorem~\ref{thm:ole} (for $\varepsilon=1$) there exist $\beta\in \qat\setminus\zat,\;\gamma\in \zat^\ast$ with $dom^+(\beta)=\emptyset=dom^-(\gamma),$ $\max dom^-(\beta)<k_1<k_2<\min dom^+(\gamma)$ and $x_0\in X$ such that $d(\mathcal{T}^{p\beta+q\gamma}(x_0),x_0)<1$ for every $0\leq p,q\leq k.$
Then, we have \begin{center}$(\ast)\;\;\;\;\;\;\;\;x_0(0)=x_0(p\beta+q\gamma)$ for every $0\leq p,q\leq k.$ \end{center}
 $x_0\in X,$ so, there exist $s_1,\ldots, s_m\in \qat$ such that $x_0$ is close to $T^{s_1}\ldots T^{s_m}\omega.$ Set $\alpha=s_1+\ldots+ s_m\in \qat.$ According to $(\ast)$ we have that $\omega(\alpha)=\omega(\alpha+p\beta+q\gamma)$ for every $0\leq p,q\leq k,$ thus, we have $\alpha+p\beta+q\gamma\in Q_{\omega(\alpha)}$ for every $0\leq p,q\leq k.$
\end{proof}

\begin{remark}
We note that Theorem~\ref{thm:n12} is not strong enough to prove Theorem~\ref{thm:ole}, mainly since we cannot have control to the support of $\alpha\in \qat$ relating to the support of $\beta\in \qat\setminus\zat,$ $\gamma\in \zat^\ast$ and since in general $\mathcal{T}^{p+q}\neq \mathcal{T}^p \mathcal{T}^q.$
\end{remark}

\section{Rational Dynamical Systems and Uniform Rational Recurrence}

 In the present paragraph we introduce the notion of a rational dynamical system (Definition~\ref{d.syst} below) and also the notion of a uniform rational recurrent point of a such system (Definition~\ref{defn:n4}). Defining the minimal rational dynamical systems (Definition~\ref{m}) we prove that every rational dynamical system has uniform rational recurrent points. In fact a rational dynamical system is minimal if and only if every point of the system is uniform rational recurrent (Theorems~\ref{thm:n6} and ~\ref{thm:n9}).

In order to define the rational dynamical systems we need the definition of the commuting rational indexed families which we give below.

\begin{defn}\label{h.syst2}
 Let $\mathcal{T}_1=(\mathcal{T}_1^q)_{q\in \qat},\ldots,\mathcal{T}_l=(\mathcal{T}_l^q)_{q\in \qat}$ be $l$ rational indexed families of $X$ defined from the commuting families of maps $\{T_{1,n}\}_{n\in \zat^\ast},\ldots,\{T_{l,n}\}_{n\in \zat^\ast}$ of $X$ respectively. We say that the families $\mathcal{T}_1,\ldots,\mathcal{T}_l$ are \textit{commuting} if $T_{i,n_1}T_{j,n_2}=T_{j,n_2}T_{i,n_1}$ for every $1\leq i,j\leq l,\;n_1,n_2\in \zat^\ast.$
\end{defn}

\begin{defn}\label{d.syst} Let $X$ be a compact metric space and $\{T_{1,n}\}_{n\in \zat^\ast},\ldots,\{T_{l,n}\}_{n\in \zat^\ast}$ be $l$ families of commuting homeomorphisms from $X$ to itself. If the indexed families $\mathcal{T}_1=(\mathcal{T}_1^q)_{q\in \qat},\ldots,\mathcal{T}_l=(\mathcal{T}_l^q)_{q\in \qat},$ which are defined from $\{T_{1,n}\}_{n\in \zat^\ast},\ldots,\{T_{l,n}\}_{n\in \zat^\ast}$ respectively are commuting, then we say that $(X,\mathcal{T}_1,\ldots,\mathcal{T}_l)$ is a \textit{rational dynamical system}.
\end{defn}

\begin{remark}\label{rem:0002} For a rational indexed family $\mathcal{T}=(\mathcal{T}^q)_{q\in \qat}$ of a rational dynamical system $(X,\mathcal{T})$ we have in general that $\mathcal{T}^{-q}\neq (\mathcal{T}^{-1})^{q},$ where we have set $(\mathcal{T}^{-1})^{q}=(\mathcal{T}^q)^{-1}.$
\end{remark}

We will define and characterize the minimal rational dynamical systems.

\begin{defn}\label{m} A rational dynamical system $(X,\mathcal{T}_1,\ldots,\mathcal{T}_l),$ where $\mathcal{T}_i=(\mathcal{T}_i^q)_{q\in \qat}$ is defined from the commuting homeomorphisms $\{T_{i,n}\}_{n\in \zat^\ast}$ of $X$ for $1\leq i\leq l,$ is said to be \textit{minimal} if for every closed $Y\subseteq X$ with $T_{i,n}(Y)= Y$ for every $1\leq i\leq l,\; n\in \zat^\ast$ we have that $Y=X$ or $Y=\emptyset.$
\end{defn}

\begin{prop}\label{charmin}
Let $(X,\mathcal{T}_1,\ldots,\mathcal{T}_l)$ be a rational dynamical system, where for ${1\leq i\leq l,}$ $\mathcal{T}_i=(\mathcal{T}_i^q)_{q\in \qat}$ is defined from the commuting homeomorphisms $\{T_{i,n}\}_{n\in \zat^\ast}.$ Let $G$ be the group of homeomorphisms of $X$ generated by the functions $T_{i,n},$ for $n\in \zat^\ast$ and $1\leq i\leq l.$ The following are equivalent:
\begin{itemize}
\item[(i)] $(X,\mathcal{T}_1,\ldots,\mathcal{T}_l)$ is minimal.
\item[(ii)] $\overline{\{S(x):\;S\in G\}}=X$ for every $x\in X.$
\item[(iii)] For every non-empty open set $V\subseteq X$ there exist a non-empty finite subset of $G,$ $F,$ such that $\bigcup_{S\in F}S^{-1}(V)=X.$
\end{itemize}
\end{prop}

\begin{proof}
$(i)\Rightarrow (ii)$ Let $x\in X.$ For every $n\in \zat^\ast,\;1\leq i\leq l,$ we have that

 $T_{i,n}(\overline{\{S(x):\;S\in G\}})=\overline{\{S(x):\;S\in G\}}.$ Since $\overline{\{S(x):\;S\in G\}}\neq \emptyset$ and $(X,\mathcal{T}_1,\ldots,\mathcal{T}_l)$ is a minimal dynamical system, we have that $\overline{\{S(x):\;S\in G\}}=X.$

 $(ii)\Rightarrow (i)$ If $Y$ is a closed non-empty closed subset of $X$ with $T_{i,n}(Y)=Y$ for every $n\in \zat^\ast,\;1\leq i\leq l,$ then $X=\overline{\{S(y):\;S\in G\}}\subseteq Y$ for every $y\in Y.$ Then $Y=X,$ thus, $(X,\mathcal{T}_1,\ldots,\mathcal{T}_l)$ is minimal.

 $(i)\Rightarrow (iii)$ For every non-empty open set $V$ we have $\bigcup_{S\in G}S^{-1}(V)=X.$ From the compactness of $X$ we have the conclusion.

 $(iii)\Rightarrow (i)$ Let $(X,\mathcal{T}_1,\ldots,\mathcal{T}_l)$ is not minimal. Let $Y$ be a non-empty closed invariant proper subset of $X$ and $V=X\setminus Y.$ Then $\bigcup_{S\in G}S^{-1}(V)\neq X,$ a contradiction.
\end{proof}

\begin{defn} Let $(X,\mathcal{T}_1,\ldots,\mathcal{T}_s)$ be a rational dynamical system and $Y\subseteq X.$ We say that the system $(Y,\mathcal{T}_1|_Y,\ldots,\mathcal{T}_s|_Y)$ is a \textit{subsystem} of $(X,\mathcal{T}_1,\ldots,\mathcal{T}_s)$ if

$(i)\;Y$ is a closed subset of $X,$ and

$(ii)\;T_{i,n}(Y)= Y$ for every $1\leq i\leq s,\;n\in \zat^\ast.$
\end{defn}

\begin{prop}\label{mds}
Every rational dynamical system has a minimal subsystem.
\end{prop}

\begin{proof}
Let $(X,\mathcal{T}_1,\ldots,\mathcal{T}_l)$ be a rational dynamical system, where for $1\leq i\leq l,$ $\mathcal{T}_i=(\mathcal{T}^q_i)_{q\in \qat}$ is defined from the commuting homeomorphisms $\{T_{i,n}\}_{n\in \zat^\ast}.$ Let $\L=\{Y\subseteq X:\;Y\neq \emptyset,\;Y$ closed and $T_{i,n}(Y)=Y$ for every $n\in \zat^\ast,\;1\leq i\leq l\}.$ $\L\neq \emptyset$ since $X\in \L.$ Let $\D\subseteq \L$ be a family totaly ordered by inclusion. $\D$ has the finite intersection property and since $X$ is compact we have that $A:=\bigcap_{Y\in \D}Y\neq\emptyset,$ with $A\subseteq Y$ for every $Y\in \D.$ According to Zorn's lemma there exists a minimal $Y_0\in \L.$ Then, $(Y_0,\mathcal{T}_1|_{Y_0},\ldots,\mathcal{T}_l|_{Y_0})$ is a minimal subsystem.
\end{proof}

We will introduce the notion of uniformly rational recurrent points for a rational dynamical system. Firstly we will remind the notion of a syndetic subset of an abelian (semi-)group.

\begin{defn}\label{defn:n3}
 A subset $E$ of an abelian (semi-)group $G$ is \textit{syndetic} if there exists $F\in [G]^{<\omega}_{>0}$ such that $G=\bigcup_{g\in F}\{s\in G:\;g+s\in E\}.$
\end{defn}

\begin{defn}\label{defn:n4}
 Let $(X,\mathcal{T}_1,\ldots,\mathcal{T}_l)$ be a rational dynamical system, where for $1\leq i\leq l,$ $\mathcal{T}_i$ is defined from the commuting homeomorphisms $\{T_{i,n}\}_{n\in \zat^\ast}.$ A point $x\in X$ is \textit{uniformly rational recurrent} for $(X,\mathcal{T}_1,\ldots,\mathcal{T}_l)$ if for any neighborhood $V$ of $x,$ the set $\{S\in G:\;S(x)\in V\}$ is syndetic, where $G$ is the group of homeomorphisms of $X$ generated by the functions $T_{i,n}$ for every $n\in \zat^\ast$ and $1\leq i\leq l.$
\end{defn}

The following theorem gives the connection between minimal rational dynamical systems and uniformly rational recurrent points.

\begin{thm}\label{thm:n6}
 If $(X,\mathcal{T}_1,\ldots,\mathcal{T}_l)$ is a minimal rational dynamical system, then every point $x\in X$ is uniformly rational recurrent.
\end{thm}

\begin{proof}
If $\mathcal{T}_i,$ for $1\leq i\leq l,$ is defined from the commuting homeomorphisms $\{T_{i,n}\}_{n\in \zat^\ast}$ and $G$ is the group of homeomorphisms of $X$ generated by the functions $T_{i,n}$ for every $n\in \zat^\ast$ and $1\leq i\leq l,$ then for every $x\in X$ and every non-empty open set $V\subseteq X,$ the set $\{S\in G:\;S(x)\in V\}$ is syndetic. Indeed, according to Proposition~\ref{charmin} we have that $\bigcup_{i=1}^{m}S_i^{-1}(V)=X$ for some $m\in \nat,$ $S_1,\ldots,S_m\in G.$ So, for every $S\in G$ we have $S_i(S(x))\in V$ for some $1\leq i\leq m,$ or $S_iS\in \{S\in G:\;S(x)\in V\}.$
\end{proof}

\begin{cor}\label{thm:n7}
For any rational dynamical system $(X,\mathcal{T}_1,\ldots,\mathcal{T}_l),$ the set of uniformly rational recurrent points is non-empty.
\end{cor}

\begin{proof}
Immediate from Proposition~\ref{mds} and Theorem~\ref{thm:n6}.
\end{proof}

Now we can characterize the minimal subsystems of a rational dynamical system via the uniformly rational recurrent points of the system.

\begin{thm}\label{thm:n9}
Let $(X,\mathcal{T}_1,\ldots,\mathcal{T}_l)$ be a rational dynamical system, where for $1\leq i\leq l,$ $\mathcal{T}_i$ is defined from the commuting homeomorphisms $\{T_{i,n}\}_{n\in \zat^\ast}$ and $G$ be the group of homeomorphisms of $X$ generated by the functions $T_{i,n}$ for every $n\in \zat^\ast$ and $1\leq i\leq l.$ Then $(Y,\mathcal{T}_1|_Y,\ldots,\mathcal{T}_l|_Y)$ is a minimal subsystem of $(X,\mathcal{T}_1,\ldots,\mathcal{T}_l)$ if and only if $Y=\overline{\{S(x):\;S\in G\}}$ for $x$ a uniformly rational recurrent point of $(X,\mathcal{T}_1,\ldots,\mathcal{T}_l).$
\end{thm}

\begin{proof}
It suffices to prove that if $y\in \overline{\{S(x):\;S\in G\}}$ then $x$ belongs to $\overline{\{S(y):\;S\in G\}}.$ Assume otherwise and let $V$ be an open neighborhood of $x$ with $\overline{V}\cap \overline{\{S(y):\;S\in G\}}=\emptyset.$ $x$ is a uniformly rational recurrent point, so there exists a finite set $\{S_1,\ldots,S_m\},\;m\in\nat$ of elements of $G$ such that for every $S\in G$ to have $S_i(S(x)) \in V$ for some $1\leq i\leq m.$ So, $\{S(x):\;S\in G\}\subseteq \bigcup_{i=1}^m S_i^{-1}(V),$ thus $y\in \overline{\{S(x):\;S\in G\}}\subseteq \bigcup_{i=1}^m S_i^{-1}(\overline{V}).$ Then we have $\overline{V}\cap \overline{\{S(y):\;S\in G\}}\neq\emptyset,$ a contradiction.
\end{proof}

\section{The recurrence properties of rational dynamical systems}

In Theorems~\ref{thm:olem1} and ~\ref{thm:olem2} below, we prove that rational dynamical systems has significant recurrence properties, analogous to those of the classical dynamical systems. So, Theorem~\ref{thm:olem1} is an extension of Theorem~\ref{thm:B} of Furstenberg-Weiss to rational dynamical systems and Theorem~\ref{thm:olem2} is an  equivalent reformulation of Theorem~\ref{thm:olem1} (for the analogous results for systems indexed by $\nat$ or $\zat$ see [Fu], [FuW] and [M]).

As a consequence of Theorem~\ref{thm:olem1}, which is a multiple version of Theorem~\ref{thm:ole} in case the transformations $T_i$ are invertible, we get a Gallai-type combinatorial result for the rational numbers (Theorem~\ref{thm:olem3}), proving that for $l\in \nat$ and any finite partition of $\qat^l,$ one of the cells of the partition contains (generalized) affine images of every finite subset of $\qat^l.$ We also remark that syndetic subsets of $\qat^l$ have the same property and that Theorem~\ref{thm:olem3} has implications for functions on large chunks of $\qat^l.$

\begin{thm}\label{thm:olem1} Let $l\in \nat$ and $(X,\mathcal{T}_1,\ldots,\mathcal{T}_l)$ a rational dynamical system, $k\in \nat$ and $k_1<k_2$ be arbitrary real numbers. For every $\varepsilon>0$ there exists $\beta\in \mathbb{Q}\setminus \zat,\;\gamma \in \zat^\ast$ with $dom^+(\beta)=\emptyset=dom^-(\gamma),$ $\max dom^-(\beta)<k_1<k_2<\min dom^+(\gamma)$ and $x_0\in X$ such that \begin{center} $\mathcal{T}_i^{p\beta +q\gamma}(x_0)\in B(x_0,\varepsilon)$ for every $0\leq p,q\leq k,\;1\leq i\leq l.$\end{center}
\end{thm}

\begin{thm}\label{thm:olem2} Let $l\in \nat,$ $(X,\mathcal{T}_1,\ldots,\mathcal{T}_l,R)$ be a minimal rational dynamical system, $k\in \nat$ and $k_1<k_2$ be arbitrary real numbers. Then for every non-empty open set $U$ there exist $\beta\in \mathbb{Q}\setminus \zat$ and $\gamma \in \zat^\ast$ with $dom^+(\beta)=\emptyset=dom^-(\gamma),$ $\max dom^-(\beta)<k_1<k_2<\min dom^+(\gamma)$ such that $$\bigcap_{0\leq p,q\leq k}(U\cap (\mathcal{T}_1^{p\beta+q\gamma})^{-1}U\cap\ldots\cap (\mathcal{T}_l^{p\beta+q\gamma})^{-1}U)\neq \emptyset.$$
\end{thm}

\begin{proof}[Proof of Theorems~\ref{thm:olem1} and ~\ref{thm:olem2}] Our method of proof is induction on $l$ and consists of three steps:

(1) We will show that Theorem~\ref{thm:olem1} holds for $l=1,$

(2) if Theorem~\ref{thm:olem1} holds for some $l\in \nat$ then Theorem~\ref{thm:olem2} also holds for $l,$ and

(3) if Theorem~\ref{thm:olem2} holds for some $l\in \nat$ then Theorem~\ref{thm:olem1} holds for $l+1.$

For $l=1$ we have the conclusion of Theorem~\ref{thm:olem1} from Theorem~\ref{thm:ole}.

Let $l\in \nat$ such that Theorem~\ref{thm:olem1} holds. Let $(X,\mathcal{T}_1,\ldots,\mathcal{T}_l,R)$ a minimal rational dynamical system, where $\mathcal{T}_i$ is defined from the commuting homeomorphisms $\{T_{i,n}\}_{n\in \zat^\ast}$ of $X$ for $1\leq i\leq l,$ $R$ is defined from the commuting homeomorphisms $\{R_{n}\}_{n\in \zat^\ast}$ of $X,k\in \nat$ and $k_1<k_2$ be arbitrary real numbers. Let $U\subseteq X$ a non-empty open set. There exists $u\in U$ and $\varepsilon>0$ such that $B(u,\varepsilon)\subseteq U.$  Let $V=B(u,\frac{\varepsilon}{2})\subseteq U.$ Then, for every $x\in X$ with $d(x,V):=\inf\{d(x,y):\;y\in V\}<\frac{\varepsilon}{2}$ we have that $x\in U.$ Let $G$ be the group of homeomorphisms generated by $\{T_{i,n}\}_{n\in \zat^\ast},\;1\leq i\leq l$ and $\{R_{n}\}_{n\in \zat^\ast}.$

Since the system $(X,\mathcal{T}_1,\ldots,\mathcal{T}_l,R)$ is minimal there exists some $m\in \nat,$ $S_1,\ldots,S_m\in G$ with $X=\bigcup^{m}_{i=1}S_i^{-1}V$ $(\ast).$
Since $X$ is compact, every $S_i,\;1\leq i\leq m$ is uniformly continuous, so there exists $\delta>0$ such that if $y,z\in X$ with $d(y,z)<\delta$ then $d(S_i(y),S_i(z))<\frac{\varepsilon}{2}$ for $1\leq i\leq m.$ According to Theorem~\ref{thm:olem1} there exist $y\in X,\;\beta\in \mathbb{Q}\setminus \zat$ and $\gamma\in \zat^\ast$ with $dom^+(\beta)=\emptyset=dom^-(\gamma),$ $\max dom^-(\beta)<k_1<k_2<\min dom^+(\gamma)$ such that $d(y,\mathcal{T}_i^{p\beta+q\gamma}(y))<\delta$ for every $0\leq p,q\leq k,\;1\leq i\leq l.$ From $(\ast)$ we have that there exists $1\leq j\leq m$ such that $y\in S_j^{-1}V.$ Set $x=S_j(y)\in V.$ Since $S_j$ commutes with the $\{T_{i,n}\}_{n\in \zat^\ast},$ for $1\leq i\leq l,$ we have that $d(x,\mathcal{T}_i^{p\beta+q\gamma}(x))<\frac{\varepsilon}{2}$ for every $0\leq p,q\leq k,\;1\leq i\leq l.$ Then we have that $\{x,\mathcal{T}_1^{p\beta+q\gamma}(x),\ldots,\mathcal{T}_l^{p\beta+q\gamma}(x)\}\subseteq U$ for every $0\leq p,q\leq k,$ so $x\in \bigcap_{0\leq p,q\leq k}(U\cap (\mathcal{T}_1^{p\beta+q\gamma})^{-1}U\cap\ldots\cap (\mathcal{T}_l^{p\beta+q\gamma})^{-1}U)\neq \emptyset,$ the conclusion.

Let that Theorem~\ref{thm:olem2} holds for some $l\in \nat.$ We will show that Theorem~\ref{thm:olem1} holds for $l+1.$ Let $(X,\mathcal{T}_1,\ldots,\mathcal{T}_{l+1})$ a rational dynamical system, where $\mathcal{T}_i$ is defined from the commuting homeomorphisms $\{T_{i,n}\}_{n\in \zat^\ast}$ of $X$ for $1\leq i\leq l+1,k\in \nat$  and $k_1<k_2$ be arbitrary real numbers. Without loss of generality we can suppose that $(X,\mathcal{T}_1,\ldots,\mathcal{T}_{l+1})$ is minimal (or else we replace $(X,\mathcal{T}_1,\ldots,\mathcal{T}_{l+1})$ with a minimal subsystem). Let $U_0$ a non-empty open set with $diam(U_0):=\sup\{d(x,y):\;x,y\in U_0\}<\frac{\varepsilon}{2}.$ According to Theorem~\ref{thm:olem2} (for the minimal system $(X,\mathcal{T}_1 \mathcal{T}^{-1}_{l+1},\ldots,\mathcal{T}_l \mathcal{T}^{-1}_{l+1},\mathcal{T}_{l+1})$) there exist $\beta_1\in \mathbb{Q}\setminus \zat,$ $\gamma_1\in \zat^\ast$ with $dom^+(\beta_1)=\emptyset=dom^-(\gamma_1),$ $\max dom^-(\beta_1)<k_1<k_2<\min dom^+(\gamma_1)$ such that $$B_0:= \bigcap_{0\leq p,q\leq k}(U_0\cap\bigcap_{s=1}^l [\mathcal{T}_s^{p\beta_1+q\gamma_1}(\mathcal{T}_{l+1}^{p\beta_1+q\gamma_1})^{-1}]^{-1}U_0)\neq \emptyset.$$ Let $U_1$ a non-empty open set with $diam(U_1)<\frac{\varepsilon}{2}$ such that \\ $U_1\subseteq \bigcap_{0\leq p,q\leq k}(\mathcal{T}_{l+1}^{p\beta_1+q\gamma_1})^{-1}B_0=\bigcap_{0\leq p,q\leq k}\bigcap_{s=1}^{l+1}(\mathcal{T}_{s}^{p\beta_1+q\gamma_1})^{-1}U_0.$

\noindent
Suppose that for $m\in \nat$ we have chosen $U_1,\ldots,U_m$ non-empty, open sets with ${diam(U_i)<\frac{\varepsilon}{2}}$ for every $1\leq i\leq m,$ such that $$(\ast\ast)\;\; U_j\subseteq \bigcap_{0\leq p,q\leq k}\bigcap_{s=1}^{l+1}(\mathcal{T}_s^{p(\beta_j+\ldots+\beta_{i+1})+q(\gamma_j+\ldots+\gamma_{i+1})})^{-1}U_i$$ for every $0\leq i<j\leq m,$  with $\beta_{j-1}+\gamma_{j-1}\prec\beta_j+\gamma_j$ for $2\leq j\leq m,$ $dom^+(\beta_j)=\emptyset=dom^-(\gamma_j)$ for $1\leq j\leq m.$ From Theorem~\ref{thm:olem2} there exist $\beta_{m+1}\in \mathbb{Q}\setminus \zat$ and $\gamma_{m+1}\in \zat^\ast$ with $\beta_{m}+\gamma_{m}\prec\beta_{m+1}+\gamma_{m+1},$ $dom^+(\beta_{m+1})=\emptyset=dom^-(\gamma_{m+1})$ such that $$B_m:=\bigcap_{0\leq p,q\leq k}(U_m\cap \bigcap_{s=1}^l [\mathcal{T}_s^{p\beta_{m+1}+q\gamma_{m+1}}(\mathcal{T}_{l+1}^{p\beta_{m+1}+q\gamma_{m+1}})^{-1}]^{-1}U_m)\neq \emptyset.$$  Let $U_{m+1}$ a non-empty open set with $diam(U_m)<\frac{\varepsilon}{2}$ and \\ $U_{m+1}\subseteq \bigcap_{0\leq p,q\leq k}(\mathcal{T}_{l+1}^{p\beta_{m+1}+q\gamma_{m+1}})^{-1}B_m=\bigcap_{0\leq p,q\leq k}\bigcap_{s=1}^{l+1}(\mathcal{T}_{s}^{p\beta_{m+1}+q\gamma_{m+1}})^{-1}U_m.$
Using this and $(\ast\ast)$ for $j=m$ we have that for every $0\leq i\leq m,$ $$U_{m+1}\subseteq\bigcap_{0\leq p,q\leq k}\bigcap_{s=1}^{l+1}(\mathcal{T}_{s}^{p(\beta_{m+1}+\ldots+\beta_{i+1})
+q(\gamma_{m+1}+\ldots+\gamma_{i+1})})^{-1}U_i.$$
Inductively we can suppose that we have sequences $(U_n)_{n\in \nat\cup\{0\}},\;(\beta_n)_{n\in \nat}$ and $(\gamma_n)_{n\in \nat}$ with $\beta_n+\gamma_n\prec\beta_{n+1}+\gamma_{n+1},$ $dom^+(\beta_n)=\emptyset=dom^-(\gamma_n)$ for every $n\in \nat,$ such that $(\ast\ast)$ holds for every $m\in \nat,$ with $\beta_j+\ldots+\beta_{i+1}\in \mathbb{Q}\setminus \zat$ and $\gamma_j+\ldots+\gamma_{i+1}\in \zat^\ast,$ for every $0\leq i<j\leq m,$ $m\in \nat.$
For every $n\in \nat\cup\{0\}$ let $x_n\in U_n.$ Since $X$ is sequential compact there exists $i_0<j_0$ such that $d(x_{i_0},x_{j_0})<\frac{\varepsilon}{2}.$ According to $(\ast\ast),$ if we set $\beta=\beta_{j_0}+\ldots+\beta_{i_{0}+1}\in \mathbb{Q}\setminus \zat$ and $\gamma=\gamma_{j_0}+\ldots+\gamma_{{i_0}+1}\in \zat^\ast,$ we have that
$\{\mathcal{T}^{p\beta+q\gamma}_1(x_{j_0}),\ldots,\mathcal{T}_{l+1}^{p\beta+q\gamma}(x_{j_0})\}\subseteq U_{i_0}$ for every $0\leq p,q\leq k.$ Also, $x_{i_0}\in U_{i_0},$ ${d(x_{i_0},x_{j_0})<\frac{\varepsilon}{2}}$ and $diam(U_{i_0})<\frac{\varepsilon}{2},$ thus for $x=x_{j_0},$ we have that \begin{center}$\mathcal{T}_i^{p\beta +q\gamma}(x)\in B(x,\varepsilon)$ for every $0\leq p,q\leq k,\;1\leq i\leq l+1.$\end{center} The proof is complete.
\end{proof}

We have already seen that Theorem~\ref{thm:ole} implies a van der Waerden-type theorem for the set of rational numbers  (Theorem~\ref{thm:n12}). Gallai extended van der Waerden theorem to higher dimensions and later Furstenberg and Weiss gave a proof of this result ([FuW]), using topological dynamics theory. Using Theorem~\ref{thm:olem1} we will state and prove, in Theorem~\ref{thm:olem3}, a Gallai-type partition theorem for the set $\qat^l,\;l\in \nat,$ using methods from [Fu].

Note that Theorem~\ref{thm:olem3} can be considered as a geometric version of Theorem~\ref{thm:olem1}.

\begin{thm}\label{thm:olem3} Let $l\in \nat$  and $k_1<k_2$ be arbitrary real numbers. If $\qat^l=Q_1\cup\ldots \cup Q_r,$ $\;r\in \nat,$ then there exists $1\leq i_0\leq r$ such that the set $Q_{i_0}$ has the property that for $k\in \nat,$ if $F\in [\qat^l]^{<\omega}_{>0},$ then there exists $\alpha\in \qat^l,\beta\in \qat\setminus \zat$ and $\gamma\in \zat^\ast$ with $dom^+(\beta)=\emptyset=dom^-(\gamma),$ $\max dom^-(\beta)<k_1<k_2<\min dom^+(\gamma)$ such that $\alpha+(p\beta+q\gamma)F\subseteq Q_{i_0}$  for every $0\leq p,q\leq k.$
\end{thm}

\begin{proof} Let $\qat^l=Q_1\cup\ldots \cup Q_r,r\in \nat.$ It suffices to produce the set $Q_j$ for a given configuration $F$ and $k\in \nat.$ For since there are only finite many possibilities for $Q_j$ and since a sequence $F_n$ may be chosen where each contains all the preceding ones and any $F$ is contained in one of them, a set $Q_j$ that occurs for infinitely many $F_n$ and $k\in \nat$ will work for all $F$ and all $k.$ That would be the desired $Q_{i_0}.$ So we assume that a finite subset of $\qat^l,\;F=\{\widetilde{e}_1,\ldots,\widetilde{e}_m\}$ and $k\in \nat$ are given. Let $\Omega=\{1,\ldots,r\}^{\qat^l}$ and enumerate $\qat=\bigcup_{n\in \nat}\{q_n\}$ with $q_1=0.$ Then $\Omega$ becomes a compact metric space if we define a metric by: \begin{center}$d(\omega,\omega')=\inf\{\frac{1}{t}:\;\omega(q_{i_1},\ldots,q_{i_l})=\omega'(q_{i_1},\ldots,q_{i_l})$ for $1\leq i_1,\ldots,i_l<t\}.$\end{center} For $1\leq i\leq m$ and $\widetilde{q}\in \qat^l,$ let \begin{center}for $n\in \nat,\;\;T_{i,n}\omega(\widetilde{q})=\omega(\widetilde{q}+(-1)^{n+1}n!\widetilde{e}_i)$ and for $n\in \zat^-,\;\;T_{i,n}\omega(\widetilde{q})=\omega(\widetilde{q}+\frac{(-1)^{-n}}{(-n+1)!}\widetilde{e}_i).$
\end{center} For $1\leq i\leq m$ we form the rational indexed family $\mathcal{T}_i$ which is defined from the commuting homeomorphisms $\{T_{i,n}\}_{n\in \zat^\ast}.$ We define a specific point $\omega\in \Omega$ by $\omega(\widetilde{q})=i\;\Leftrightarrow\; \widetilde{q}\in Q_i$ and $\widetilde{q}\notin Q_j$ for $j<i$ and let \begin{center}$X=\overline{\{\mathcal{T}_1^{s_{1,1}}\ldots \mathcal{T}_1^{s_{1,l_1}}\ldots \mathcal{T}_m^{s_{m,1}}\ldots \mathcal{T}_m^{s_{m,l_m}}\omega,\;s_{i,j}\in \qat,\;l_i\in \nat,\;1\leq j\leq l_i,\;1\leq i\leq m\}}.$\end{center} Then, $(X,\mathcal{T}_1,\ldots,\mathcal{T}_m)$ is a rational dynamical system, so, according to Theorem~\ref{thm:olem1} (for $\varepsilon=1$) there exists $\beta\in \qat\setminus \zat,\;\gamma\in \zat^\ast$ with $dom^+(\beta)=\emptyset=dom^-(\gamma),$ $\max dom^-(\beta)<k_1<k_2<\min dom^+(\gamma)$ and $x_0\in X$ such that $d(\mathcal{T}_i^{p\beta+q\gamma}(x_0),x_0)<1$ for every ${0\leq p,q\leq k,}$ $1\leq i\leq m.$
 Thus \begin{center} $(\ast)\;\;\;x_0((0,\ldots,0))=x_0((p\beta+q\gamma)\widetilde{e}_1)=\ldots=
x_0((p\beta+q\gamma)\widetilde{e}_m)$ for every $0\leq p,q\leq k.$\end{center}  $x_0\in X,$ so it is arbitrary close to some translate of $\omega,$ $\mathcal{T}_1^{s_{1,1}}\ldots \mathcal{T}_1^{s_{1,l_1}}\ldots \mathcal{T}_m^{s_{m,1}}\ldots \mathcal{T}_m^{s_{m,l_m}}\omega,$ for some $s_{i,j}\in \qat,\;l_i\in \nat,\;1\leq j\leq l_i,\;1\leq i\leq m.$

 Let $\widetilde{\alpha}=(s_{1,1}+\ldots +s_{1,l_1})\widetilde{e}_1+\ldots +(s_{m,1}+\ldots+s_{m,l_m})\widetilde{e}_m.$
It follows from $(\ast)$ that \begin{center} $\omega(\widetilde{\alpha})=\omega(\widetilde{\alpha}+(p\beta+q\gamma)\widetilde{e}_1)=\ldots=
\omega(\widetilde{\alpha}+(p\beta+q\gamma)\widetilde{e}_m)$ for every $0\leq p,q\leq k,$\end{center} so, we have $\widetilde{\alpha}+(p\beta+q\gamma)F\subseteq Q_{\omega(\widetilde{\alpha})}$ for every $0\leq p,q\leq k.$
\end{proof}

\begin{remark}\label{rem:002} According to the proof of the previous theorem, we see that we cannot have control to the domain of the coordinates of $\alpha\in \qat^l$ (as in the proof of Theorem~\ref{thm:ole}). For this reason and also since for a rational indexed family $\mathcal{T},$ we have in general that $\mathcal{T}^{p+q}\neq \mathcal{T}^p \mathcal{T}^q$ ($p,q\in \qat$) efforts to prove Theorem~\ref{thm:olem1} or ~\ref{thm:olem2} from Theorem~\ref{thm:olem3} (proving simultaneously the equivalence of these theorems) were fruitless.
\end{remark}

Let give some corollaries of Theorem~\ref{thm:olem3}.

\begin{defn}\label{defn:n10}
Let $l\in \nat$ and $k_1<k_2$ be arbitrary real numbers. We say that the subset $B\subseteq \qat^l$ is a \textit{RVDW(l,k$_1$,k$_2$)}-set if for every $F\in [\qat^l]^{<\omega}_{>0}$ and $k\in \nat$ there exist $\alpha\in \qat^l,\beta\in \qat\setminus\zat$ and $\gamma\in \zat^\ast$ with $dom^+(\beta)=\emptyset=dom^-(\gamma),$ $\max dom^-(\beta)<k_1<k_2<\min dom^+(\gamma)$ such that $\alpha+(p\beta+q\gamma)F\subseteq B$ for every $0\leq p,q\leq k.$
\end{defn}

We will now prove that syndetic sets belongs to the previous family.

\begin{cor}\label{prop:n11}
Let $l\in \nat$ and $k_1<k_2$ be arbitrary real numbers. If $E$ is a syndetic subset of $\qat^l,$ then $E$ is a RVDW(l,k$_1$,k$_2$)-set.
\end{cor}

\begin{proof} Let $E$ be a syndetic subset of $\qat^l$ and $k_1<k_2$ be arbitrary real numbers. Then, $\qat^l=\bigcup_{x\in F}(E+x)$ for some $F\in [\qat^l]^{<\omega}_{>0}.$ According to Theorem~\ref{thm:olem3} there exists $x_0\in F$ such that $E+x_0$ is a RVDW($l,k_1,k_2$)-set. So, $E$ is a RVDW($l,k_1,k_2$)-set, since this property is translation invariant.
\end{proof}

Theorem~\ref{thm:olem3} has implications for functions on large chunks of $\qat^l.$

\begin{cor}\label{thm:n13}
Let $F\in [\qat^l]^{<\omega}_{>0},\;l\in \nat,\;k_1<k_2$ be arbitrary real numbers, $k\in \nat$ and $r\in \nat.$  There exist $n_0\equiv n_0(l,k,k_1,k_2,r,F)\in \nat$ such that, if $\qat^l_n,\;n\in \nat$ denotes the set of vectors in $\qat^l$ with components between $-n$ and $n,$ then whenever $n\geq n_0$ and $\qat^l_n=Q_1\cup\ldots \cup Q_r,$ there exist $1\leq i_0\leq r,\;\alpha\in \qat^l,\;\beta\in \qat\setminus\zat$ and $\gamma\in \zat^\ast$ with $dom^+(\beta)=\emptyset=dom^-(\gamma),$ $\max dom^-(\beta)<k_1<k_2<\min dom^+(\gamma)$ such that \begin{center} $\alpha+(p\beta+q\gamma)F\subseteq Q_{i_0}$ for every $0\leq p,q\leq k.$ \end{center}
\end{cor}

\begin{proof}
Suppose that for $n\rightarrow \infty$ we can find partitions such that the conclusion doesn't hold. Consider the corresponding functions from $\qat^l_n$ to $\Lambda=\{1,\ldots,r\}$ which are defined from these partitions and extend them arbitrary to $\qat^l$ to obtain a point $\omega_n\in \Lambda^{\qat^l}.$ Let $\omega$ a limit point of $(\omega_n)_{n\in \nat}.$ According to Theorem~\ref{thm:olem3} there exist $\alpha\in \qat^l,\;\beta\in \qat\setminus\zat$ and $\gamma\in \zat^\ast$ with $dom^+(\beta)=\emptyset=dom^-(\gamma),$ $\max dom^-(\beta)<k_1<k_2<\min dom^+(\gamma)$ such that $\omega$ is constant on $\alpha+(p\beta+q\gamma)F$ for every $0\leq p,q\leq k.$ For $n$ large enough, we have that $\alpha+(p\beta+q\gamma)F\subseteq Q_n^l$ for every $0\leq p,q\leq k$ and that for $n'$ large $\omega_{n'}|_{Q_n^l}=\omega|_{Q_n^l},$ a contradiction.
\end{proof}

\subsection*{Notes} Since (as we already noticed in Remark~\ref{rem:000} and Remark~\ref{rem:002}) for a rational indexed family $\mathcal{T}=(\mathcal{T}^q)_{q\in \qat},$ we have in general that $\mathcal{T}^{p+q}\neq \mathcal{T}^p \mathcal{T}^q$ ($p,q\in \qat$), we can't have (with these methods) polynomial extensions of the results of this paragraph.

\section{A Furstenberg-Weiss-type Theorem for rational dynamical systems}

In this section we will prove (in Theorem~\ref{thm:olem4}) a strengthening of Theorem~\ref{thm:A} for rational dynamical systems, namely, we prove that if $(X,\mathcal{T}_1,\ldots,\mathcal{T}_l)$ is a rational dynamical system, $k\in \nat$ and $k_1<k_2$ are arbitrary real numbers that there exist $x\in X$ and sequences $(\beta_n)_{n\in \nat}\subseteq \qat\setminus\zat,$ $(\gamma_n)_{n\in \nat}\subseteq \zat^\ast$ with $dom^+(\beta_1)=\emptyset=dom^-(\gamma_1),$ $\max dom^-(\beta_1)<k_1<k_2<\min dom^+(\gamma_1),$ $\beta_n+\gamma_n\prec\beta_{n+1}+\gamma_{n+1},$ $dom^+(\beta_n)=\emptyset=dom^-(\gamma_n)$ for every $n\in \nat$ such that $\mathcal{T}_i^{p\beta_n+q\gamma_n}(x)\rightarrow x$ for every $0\leq p,q\leq k$ simultaneously for $1\leq i\leq l$ (we call these points \textit{multiple rational recurrent points}). Moreover we prove that Theorem~\ref{thm:olem4} is equivalent to Theorems~\ref{thm:olem1} and ~\ref{thm:olem2} and also that the multiple rational recurrent points consist a residual subset of $X$ (Definition~\ref{defn:n1}).

At this point (as in [Fu] for the analogous dynamical systems related to $\nat$ or $\zat$) observe that if the condition of commutativity of the system is omitted, the conclusion does not holt. For example, let $X=\mathbb{R}\cup\{\infty\},\;T_n(x)=\frac{x}{2},\;n\in \zat^\ast$ and $S_n(x)=x+1,\;n\in \zat^\ast.$ If $\mathcal{T}$ and $\mathcal{S}$ are defined from $\{T_n\}_{n\in \zat^\ast}$ and $\{S_n\}_{n\in \zat^\ast}$ respectively, then, the only recurrent point of $\mathcal{T}$ is $0$ and the only one for $\mathcal{S}$ is $\infty.$ Also, without commutativity it may still happen that the return times of any point are disjoint for the various transformations. For example, let $X=\{-1,1\}^{\qat}$ and $\mathcal{T}$ be the indexed family which is defined from $\{T_n\}_{n\in \zat^\ast}$ with \begin{center} $T_{n}\omega(q)=\omega(q+(-1)^{n+1}n!),\;n\in \nat$ and $T_{n}\omega(q)=\omega(q+\frac{(-1)^{-n}}{(-n+1)!}),\;n\in \zat^-.$
\end{center}
Let $R:X\rightarrow X$ with $R(\omega(q))=\omega(q)$ if $q=0$ and $R(\omega(q))=-\omega(q)$ if $q\neq0$ and let $S_n=RT_nR,\;n\in \zat^\ast.$ Then, if $\mathcal{S}$ is the indexed family which is defined from $\{S_n\}_{n\in \zat^\ast},$ we have that $\mathcal{S}^q=R\mathcal{T}^qR$ for every $q\in \qat.$ $\mathcal{T}^q\omega$ close to $\omega$ implies that $\omega(q)=\omega(0).$ $\mathcal{S}^q\omega$ close to $\omega$ implies that $\mathcal{T}^qR\omega$ is close to $R\omega,$ so $R\omega(q)=R\omega(0),$ thus $-\omega(q)=\omega(0)$ if $q\neq 0.$ We have that $\mathcal{T}^q\omega$ and $\mathcal{S}^q\omega$ cannot be simultaneously close to $\omega.$

\medskip

A strengthening of Theorem~\ref{thm:A} related to rational numbers is the following.

\begin{thm}\label{thm:olem4} Let $(X,\mathcal{T}_1,\ldots,\mathcal{T}_l)$ be a rational dynamical system, $k\in \nat$ and $k_1<k_2$ be arbitrary real numbers. There exists a $x\in X$ and sequences $(\beta_n)_{n\in \nat}\subseteq \qat\setminus\zat,\;(\gamma_n)_{n\in \nat}\subseteq \zat^\ast$ with $dom^+(\beta_1)=\emptyset=dom^-(\gamma_1),$ $\max dom^-(\beta_1)<k_1<k_2<\min dom^+(\gamma_1),$ $\beta_n+\gamma_n\prec\beta_{n+1}+\gamma_{n+1},$ $dom^+(\beta_n)=\emptyset=dom^-(\gamma_n)$ for every $n\in \nat$ such that \begin{center} $\mathcal{T}_i^{p\beta_n+q\gamma_n}(x)\rightarrow x$ for every $0\leq p,q\leq k$ simultaneously for $1\leq i\leq l.$ \end{center}
\end{thm}

\begin{proof} For every $s>0$ let

\begin{center} $F_s=\{x\in X:$ there exists $\beta\in \qat\setminus\zat,\gamma\in \zat^\ast$ with $dom^+(\beta)=\emptyset=dom^-(\gamma),\;\;\;\;\;\;\;\;\;\;\;\;\;\;\;\;\;\;$\\ $\max dom^-(\beta)<k_1<k_2<\min dom^+(\gamma)$ such that $d(\mathcal{T}_{i}^{p\beta+q\gamma}(x),x)< \frac{1}{s}\;\;\;\;\;\;\;$\\ for every $0\leq p,q\leq k,$ $1\leq i\leq l\}.\;\;\;\;\;\;\;\;\;\;\;\; \;\;\;\;\;\;\;\;\;\;\;\; \;\;\;\;\;\;\;\;\;\;\;\; \;\;\;\;\;\;\;\;\;\;\;\; \;\;\;\;\;\;\;\;\;\;\;\;\;\;\;\;$\end{center}

\noindent If the conclusion is not true, then $X=\bigcup_{n\in \nat}(X\setminus F_n).$ We claim that for every $n\in \nat$ we have that $(X\setminus F_n)^{\circ}=\emptyset,$ a contradiction according to Baire's Category Theorem, since every $X\setminus F_n$ is closed.

Suppose that $(X\setminus F_{n_0})^{\circ}\neq \emptyset$ for some $n_0\in \nat.$ Without loss of generality we can suppose that $(X,\mathcal{T}_1,\ldots,\mathcal{T}_l)$ is minimal. For $1\leq i\leq l,$ let $\mathcal{T}_i$ be defined from the commuting homeomorphisms $\{T_{i,n}\}_{n\in \zat^\ast}$ of $X.$ If $G$ is the group of homeomorphisms generated by $\{T_{i,n}\}_{n\in \zat^\ast},\;1\leq i\leq l,$ then $X=S_1^{-1}(X\setminus F_{n_0})^{\circ}\cup\ldots\cup S_m^{-1}(X\setminus F_{n_0})^{\circ}$ for some $S_1,\ldots,S_m\in G,m\in \nat.$ Choose $\delta>0$ such that if $y,z\in X$ with $d(y,z)<\delta$ then $d(S_i(y),S_i(z))<\frac{1}{n_0}$ for every $1\leq i\leq m.$

We claim that if $x\in S_j^{-1}(X\setminus F_{n_0})^{\circ}$ for some $1\leq j\leq m,$ then $x\in X\setminus F_{\frac{1}{\delta}}.$ Indeed, if there are $\beta\in \qat\setminus\zat,\gamma\in \zat^\ast$ with $dom^+(\beta)=\emptyset=dom^-(\gamma),$ $\max dom^-(\beta)<k_1<k_2<\min dom^+(\gamma)$ such that $d(\mathcal{T}_{i}^{p\beta+q\gamma}(x),x)< \delta$ for every $0\leq p,q\leq k,$ $1\leq i\leq l,$ then  $d(S_j(\mathcal{T}_{i}^{p\beta+q\gamma}(x)),S_j(x))=d(\mathcal{T}_{i}^{p\beta+q\gamma}(S_j(x)),S_j(x))<\frac{1}{n_0}$ for every $0\leq p,q\leq k,$ $1\leq i\leq l,$ since $S_j$ commutes with $\{T_{i,n}\}_{n\in \zat^\ast},1\leq i\leq l.$ Since $S_j(x)\in (X\setminus F_{n_0})^{\circ},$ we have a contradiction.

Since every $x\in X$ is in $S_j^{-1}(X\setminus F_{n_0})^{\circ}$ for some $1\leq j\leq m,$ we have proved that $x\in X\setminus F_{\frac{1}{\delta}}$ for every $x\in X,$ so, $X\setminus F_{\frac{1}{\delta}}=X$ a contradiction according to Theorem~\ref{thm:olem1}.
\end{proof}

\begin{defn}\label{defn:n1}
A subset $U\subseteq X$ is called \textit{residual} if it contains an enumerable intersection of dense sets.
\end{defn}

Theorem~\ref{thm:olem4} gives the following.

\begin{prop}\label{prop:n2}
 If $(X,\mathcal{T}_1,\ldots,\mathcal{T}_l)$ is a minimal rational dynamical system, then the set of multiple recurrent points of $X$ is residual.
\end{prop}

\begin{proof}
It follows from Theorem~\ref{thm:olem4}, since for every $n\in \nat$ the set $F_n$ is dense, open and $\emptyset\neq F_n\subseteq\{$multiple rational recurrent points$\}.$
\end{proof}

\begin{prop}\label{prop:002} Theorems~\ref{thm:olem1},~\ref{thm:olem2} and ~\ref{thm:olem4} are equivalent.
\end{prop}

\begin{proof} We have already seen that Theorems~\ref{thm:olem1} and ~\ref{thm:olem2} are equivalent and that Theorem~\ref{thm:olem4} follows from Theorem~\ref{thm:olem1}. Let $l\in \nat,(X,\mathcal{T}_1,\ldots,\mathcal{T}_l,R)$ be a minimal rational dynamical system, $k\in \nat$ and $k_1<k_2$ be arbitrary real numbers. According to the proof of Theorem~\ref{thm:olem4}, if $U$ is a non-empty set in $X$ then there exist a $x\in U$ and sequences $(\beta_n)_{n\in \nat}\subseteq \qat\setminus\zat,\;(\gamma_n)_{n\in \nat}\subseteq \zat^\ast$ with $dom^+(\beta_1)=\emptyset=dom^-(\gamma_1),$ $\max dom^-(\beta_1)<k_1<k_2<\min dom^+(\gamma_1),$ $\beta_n+\gamma_n\prec\beta_{n+1}+\gamma_{n+1},$ $dom^+(\beta_n)=\emptyset=dom^-(\gamma_n)$ for every $n\in \nat$ such that $\mathcal{T}_i^{p\beta_n+q\gamma_n}(x)\rightarrow x$ for every $0\leq p,q\leq k,$ $1\leq i\leq l.$ So, there exists $n_0\in \nat$ such that $\mathcal{T}_i^{p\beta_{n_0}+q\gamma_{n_0}}(x)\in U$ for every $0\leq p,q\leq k,$ $1\leq i\leq l.$

Hence $\;\;\;\bigcap_{0\leq p,q\leq k}(U\cap (\mathcal{T}_1^{p\beta_{n_0}+q\gamma_{n_0}})^{-1}U\cap\ldots\cap (\mathcal{T}_l^{p\beta_{n_0}+q\gamma_{n_0}})^{-1}U)\neq \emptyset,$ the conclusion.
\end{proof}

\begin{defn}\label{defn:n15}
Let $k\in \nat$ and $k_1<k_2$ arbitrary real numbers. A subset $A\subseteq \qat$ is called  \textit{RIP(k,k$_1$,k$_2$)-set} if there exist sequences $(\beta_n)_{n\in \nat}\subseteq \qat\setminus\zat,$ $(\gamma_n)_{n\in \nat}\subseteq \zat^\ast$ with $dom^+(\beta_1)=\emptyset=dom^-(\gamma_1),$ $\max dom^-(\beta_1)<k_1<k_2<\min dom^+(\gamma_1),$ $\beta_n+\gamma_n\prec\beta_{n+1}+\gamma_{n+1},$ $dom^+(\beta_n)=\emptyset=dom^-(\gamma_n)$ for every $n\in \nat$ such that $A$ consists of the numbers $p\beta_i+q\gamma_i,\;0\leq p,q\leq k$ together with all the finite sums in the form

 $p(\beta_{i_1}+\ldots+\beta_{i_s})+q(\gamma_{i_1}+\ldots+\gamma_{i_s}),\;0\leq p,q\leq k$ with $i_1<\ldots<i_s.$
\end{defn}

\begin{prop}\label{thm:n14}
Let $(X,\mathcal{T}_1,\ldots,\mathcal{T}_l)$ a rational dynamical system, $k\in \nat,k_1<k_2$ arbitrary real numbers and $x_0$ a multiple rational recurrent point of $X.$ Then, for every $\delta>0$ the set  $R_{\delta}=\{q\in \qat:\;d(\mathcal{T}^q_i(x_0),x_0)<\delta$ for every $1\leq i\leq l\}$ contains a RIP(k,$k_1,k_2$)-set.
\end{prop}

\begin{proof}
Let $\delta>0$ and $x_0$ a point which satisfies the conclusion of Theorem~\ref{thm:olem4}. According to Theorem~\ref{thm:olem4} there exist $\beta_1\in \qat\setminus\zat$ and $\gamma_1\in \zat^\ast$ with $dom^+(\beta_1)=\emptyset=dom^-(\gamma_1),$ $\max dom^-(\beta_1)<k_1<k_2<\min dom^+(\gamma_1)$ such that \\ $(1)\;\;\;\;d(\mathcal{T}^{p\beta_1+q\gamma_1}_i(x_0),x_0)<\delta$ for every $0\leq p,q\leq k,$ $1\leq i\leq l.$

Let $0<\delta_2\leq \delta$ such that \\
$(2)\;\;\;d(x,x_0)<\delta_2\Rightarrow d(\mathcal{T}^{p\beta_1+q\gamma_1}_i(x),x_0)<\delta$ for every $0\leq p,q\leq k,$ $1\leq i\leq l.$

According to Theorem~\ref{thm:olem4} there exist $\beta_2\in \qat\setminus\zat$ and $\gamma_2\in \zat^\ast$ with $\beta_1+\gamma_1\prec \beta_2+\gamma_2,$ $dom^+(\beta_2)=\emptyset=dom^-(\gamma_2)$ such that \\ $(3)\;\;\;\;d(\mathcal{T}^{p\beta_2+q\gamma_2}_i(x_0),x_0)<\delta_2$ for every $0\leq p,q\leq k,$ $1\leq i\leq l.$

The conditions $(1),(2)$ and $(3)$ ensures that

\noindent$(\ast)\;\;\;d(\mathcal{T}^m_i(x_0),x_0)<\delta,$ where $m=p\beta_1+q\gamma_1$ or $p\beta_2+q\gamma_2$ or $p(\beta_1+\beta_2)+q(\gamma_1+\gamma_2)$ for every $0\leq p,q\leq k,$ $1\leq i\leq l.$

Assume that we have found $\beta_1,\ldots,\beta_n,\;\gamma_1,\ldots,\gamma_n$ with $\beta_s+\gamma_s\prec \beta_{s+1}+\gamma_{s+1}$ for every $s=1,\ldots,n-1,$ $dom^+(\beta_s)=\emptyset=dom^-(\gamma_s),$ $1\leq s\leq n$   such that $(\ast)$ holds for $m=p(\beta_{i_1}+\ldots+\beta_{i_s})+q(\gamma_{i_1}+\ldots+\gamma_{i_s})$ for every $0\leq p,q\leq k,$ $1\leq i\leq l$ with $i_1<\ldots<i_s\leq n.$ Let $\delta_{n+1}\leq \delta$ such that $d(x,x_0)<\delta_{n+1}\;\Rightarrow\;d(\mathcal{T}^m_i(x),x_0)<\delta$ for the previous $m,\;1\leq i\leq l.$
According to Theorem~\ref{thm:olem4} there exist $\beta_{n+1}\in \qat\setminus\zat$ and $\gamma_{n+1}\in \zat^\ast$ with $\beta_n+\gamma_n\prec \beta_{n+1}+\gamma_{n+1},$ $dom^+(\beta_{n+1})=\emptyset=dom^-(\gamma_{n+1})$ such that \\ $\;\;\;\;d(\mathcal{T}^{p\beta_{n+1}+q\gamma_{n+1}}_i(x_0),x_0)<\delta_{n+1}$ for every $0\leq p,q\leq k,\;1\leq i\leq l.$ Then, $(\ast)$ holds if we replace $m$ with $m+p\beta_{n+1}+q\gamma_{n+1}$ or $p\beta_{n+1}+q\gamma_{n+1}$ for every $0\leq p,q\leq k.$ Inductively, we have that

$R=\{p(\beta_{i_1}+\ldots+\beta_{i_s})+q(\gamma_{i_1}+\ldots+\gamma_{i_s}),\;0\leq p,q\leq k$ with $i_1<\ldots<i_s\}\subseteq R_{\delta}.$
\end{proof}

\section{Some Applications}

In this last section we will present some applications of Theorem~\ref{thm:olem3} not only to topology but also to diophantine approximations and number theory.

A finite partition of $\qat^l,l\in \nat$ can be considered as a function from $\qat^l$ to a finite set. Analogously to [Fu] (Theorem 2.9 and Lemma 2.11) we extend Theorem~\ref{thm:olem3} from the finite partitions of $\qat^l$ to functions from $\qat^l$ to a compact metric space.

\begin{thm}\label{thm:cor2} Let $l\in \nat,k_1<k_2$ be arbitrary real numbers and let $f:\qat^l\rightarrow X$ be an arbitrary function with values on the compact metric space $X.$ For any $\varepsilon>0,k\in \nat$ and $F\in [\qat^l]^{<\omega}_{>0}$ we can find $\alpha\in \qat^l,\beta\in \qat\setminus \zat$ and $\gamma\in \zat^\ast$ with $dom^+(\beta)=\emptyset=dom^-(\gamma),$ $\max dom^-(\beta)<k_1<k_2<\min dom^+(\gamma)$ such that $diam (f(\alpha+(p\beta+q\gamma)F))<\varepsilon$  for every $0\leq p,q\leq k.$
\end{thm}

\begin{proof}
Let $X=\bigcup^r_{i=1}U_i,r\in \nat$ where $diam (U_i)<\varepsilon$ for every $1\leq i\leq r.$ Then $\qat^l=\bigcup^r_{i=1}f^{-1}(U_i).$ We obtain the result applying Theorem~\ref{thm:olem3} to this partition.
\end{proof}

\medskip

We will now give an application of Theorem~\ref{thm:cor2} to diophantine approximations.

\medskip

 Let $\delta$ be an arbitrary real number and $f(q)=e^{i\pi q^2 \delta},\;q\in \qat.$ According to Theorem~\ref{thm:cor2}, for $\varepsilon>0,$ $k\in \nat$ and $k_1<k_2$ arbitrary real numbers there exist $\alpha\in \qat,\beta\in \qat\setminus\zat$ and $\gamma\in \zat^\ast$ with $dom^+(\beta)=\emptyset=dom^-(\gamma),$ $\max dom^-(\beta)<k_1<k_2<\min dom^+(\gamma)$ such that $|f(\alpha)-f(\alpha+(p\beta+q\gamma))|<\varepsilon$ and $|f(\alpha)-f(\alpha+2(p\beta+q\gamma))|<\varepsilon$ for every $0\leq p,q\leq k.$ If $h\equiv h(p,q)=p\beta+q\gamma,$ we have $\;|1-e^{i(2\alpha h+h^2)\pi \delta}|<\varepsilon$ and $|1-e^{i(4\alpha h+4h^2)\pi \delta}|<\varepsilon.$
  Since $\cos(x)\leq 1$ for every $x,$ we have $|1-e^{2ix}|\leq 2|1-e^{ix}|.$

   $2h^2 \pi \delta=[(4\alpha h+4h^2)-2(2\alpha h+h^2)]\pi\delta,$ so \begin{center} $|1-e^{2ih^2\pi\delta}|\leq 2|1-e^{i(2\alpha h+h^2)\pi\delta}|+|1-e^{i(4\alpha h+4h^2)\pi\delta}|<3\varepsilon.$ \end{center}
If we set $\xi\equiv \xi(p,q)=h^2\delta-[h^2\delta]$ for every $0\leq p,q\leq k,$ we have $|1-e^{2\pi i\xi}|<3\varepsilon\;\Leftrightarrow\;2\sin(\pi\xi)<3\varepsilon.$
Using the inequality $\sin(x)\geq \frac{2}{\pi}x$ for $0\leq x\leq\frac{\pi}{2},$ we get:

$(i)$ If $\pi\xi\in [0,\frac{\pi}{2}],$ then for $m\equiv m(p,q)=[h^2\delta]$ we have $|h^2\delta-m|<\varepsilon.$

$(ii)$ If $\pi\xi\in (\frac{\pi}{2},\pi)$ then $\pi-\pi\xi\in (0,\frac{\pi}{2}),$ so, for $m\equiv m(p,q)=[h^2\delta]+1$ we have $|h^2\delta-m|<\varepsilon.$

This implies that there exists $m(p,q)\in \zat$ such that $|\delta(p\beta+q\gamma)^2-m(p,q)|<\varepsilon$ for every $0\leq p,q\leq k.$

 Note that (for $p=0$) we have that for every $\delta$ real number we can solve $|\delta n^2-m|<\varepsilon$ for every $\varepsilon>0$ (a result first proved by Hardy and Littlewood).

\medskip

We will now state and prove the multidimensional version of Theorem~\ref{thm:cor2}.

\begin{thm}\label{cor:olem2} Let $l_1,\ldots,l_s\in \nat,s\in \nat$ and $k_1<k_2$ be arbitrary real numbers. Let $f_1:\qat^{l_1}\rightarrow X_1,\ldots,f_s:\qat^{l_s}\rightarrow X_s$ be $s$ arbitrary functions with values on the compact metric spaces $X_1,\ldots,X_s$ respectively. For any $\varepsilon>0,k\in \nat$ and $F_1\in [\qat^{l_1}]^{<\omega}_{>0},\ldots,F_s\in [\qat^{l_s}]^{<\omega}_{>0}$ we can find $\alpha_i\in \qat^{l_i},1\leq i\leq s,$ $\beta\in \qat\setminus \zat$ and $\gamma\in \zat^\ast$ with $dom^+(\beta)=\emptyset=dom^-(\gamma),$ $\max dom^-(\beta)<k_1<k_2<\min dom^+(\gamma)$ such that $diam (f_i(\alpha_i+(p\beta+q\gamma)F_i))<\varepsilon$  for every $0\leq p,q\leq k,$ $1\leq i\leq s.$
\end{thm}

\begin{proof}
Form $f_1\times\ldots\times f_s:\qat^{l_1+\ldots+l_s}\rightarrow X_1\times\ldots \times X_s,$ $F=F_1\times\ldots\times F_s$ and apply Theorem~\ref{thm:cor2}.
\end{proof}

\medskip

Let give an application of Theorem~\ref{cor:olem2} to number theory.

\medskip

 Let $\pi(x)$ be a real polynomial with $\pi(0)=0,\;k\in \nat$ and $k_1<k_2$ arbitrary real numbers. We will show that for every $\varepsilon>0$ there exist $\beta\in\qat\setminus\zat,\;\gamma\in \zat^\ast$ with $dom^+(\beta)=\emptyset=dom^-(\gamma),$ $\max dom^-(\beta)<k_1<k_2<\min dom^+(\gamma)$ and integers $m(p,q),$ $0\leq p,q\leq k$ such that $|\pi(p\beta+q\gamma)-m(p,q)|<\varepsilon$ for every $0\leq p,q\leq k.$

Let $\pi(x)=b_n x^n+\ldots+b_1 x.$ We write $\pi(x)=s_n A_n x^n+\ldots+s_1 A_1 x,$ with $A_r=\sum^r_{j=0}(-1)^j \binom{r}{j} j^r,\;r=1,\ldots,n$ (where $\binom{r}{j}=\frac{r!}{j!(r-j)!}$).

 For every $r=1,\ldots,n$ we set $f_r:\qat\rightarrow \mathbb{R}/\zat$ with $f_r(q)=s_r q^r.$ According to Theorem~\ref{cor:olem2} there exist $\alpha_1,\ldots,\alpha_n\in \qat,\;\beta\in \qat\setminus\zat,\;\gamma\in \zat^\ast$ with $dom^+(\beta)=\emptyset=dom^-(\gamma),$ $\max dom^-(\beta)<k_1<k_2<\min dom^+(\gamma)$ such that \begin{center} $\|f_r(\alpha_r+j(p\beta+q\gamma))-f_r(\alpha_r)\|<\frac{\varepsilon}{2^{n+1}},\;j=1,\ldots,r,\;0\leq p,q\leq k.$ \end{center}

 We can easily prove by induction that $\sum^r_{j=0}(-1)^j\binom{r}{j}(x+jy)^r=A_r y^r,$ so, we have that  $\sum^r_{j=0}(-1)^j\binom{r}{j}f_r(\alpha_r+j(p\beta+q\gamma))=A_r f_r(p\beta+q\gamma)$ for every $0\leq p,q\leq k.$ Since, \begin{center} $\sum^r_{j=0}(-1)^j\binom{r}{j}f_r(\alpha_r)=0$ and $\sum^r_{j=0}\binom{r}{j}=2^r,$ \end{center} we have that $\|A_r f_r(p\beta+q\gamma)\|<\frac{2^r\varepsilon}{2^{n+1}},\;r=1,\ldots,n,$ thus \begin{center} $\|\pi(p\beta+q\gamma)\|<\frac{\varepsilon}{2^{n+1}}(\sum^n_{r=1}2^r)<\varepsilon$ for every $0\leq p,q\leq k.$\end{center}

Note that (for $p=0$) for every real polynomial $\pi(x)$ with $\pi(0)=0$ and $\varepsilon>0,$ we can find integers $m,n$ with $|\pi(n)-m|<\varepsilon$ (a result first proved by Hardy and Weyl).

\medskip

\subsection*{Acknowledgments.} The author wish to thank Professors V. Farmaki and S. Negrepontis for helpful discussions and support during the preparation of this paper. The author also acknowledge partial support from the State Scholarship Foundation of Greece.

{\footnotesize

\noindent
\newline
Andreas Koutsogiannis:
\newline
{\sc Department of Mathematics, Athens University, Panepistemiopolis, 15784 Athens, Greece}
\newline
E-mail address: akoutsos@math.uoa.gr}
\end{document}